\documentclass[12pt]{amsart}

\usepackage{amsmath,amsthm,amssymb}
\usepackage{enumitem}
\usepackage{tikz}
\usetikzlibrary{positioning}
\usepackage{hyperref}  
\usepackage{mathabx}
\usepackage{amsfonts,amscd}
\usepackage{mathtools}
\usepackage{wasysym} 
\usepackage{graphicx}
\usepackage{psfrag}
\usepackage{datetime}
\usepackage{tikz-cd}  
\usepackage{geometry} 
\usepackage[all]{xy}

\usepackage{tikz,tikz-cd}  
\usetikzlibrary{positioning}

\usepackage{geometry} 
\usepackage[all]{xy}

\usepackage{hyperref}  

\title{Quiddity sequences for $\mathrm{SL}_3$-frieze patterns}
\author{Jordan McMahon}

\begin{document}

\newtheorem{lm}{Lemma}[section]
\newtheorem{prop}[lm]{Proposition}
\newtheorem{conj}[lm]{Conjecture}
\newtheorem{eg}[lm]{Example}
\newtheorem*{claim}{Claim}
\newtheorem{satz}[lm]{Satz}

\newtheorem*{corollary}{Corollary}
\newtheorem{cor}[lm]{Corollary}
\newtheorem{theorem}[lm]{Theorem}
\newtheorem*{thm}{Theorem}

\theoremstyle{definition}
\newtheorem{defi}[lm]{Definition}
\newtheorem{defn}[lm]{Definition}

\newtheorem*{defini}{Definition}
\newtheorem*{definitionen}{Definitionen}
\newtheorem{ueb}{\"Ubungsbeispiel}
\newtheorem{bsp}[lm]{Beispiel}
\newtheorem{bspe}[lm]{Beispiele}
\newtheorem*{ex}{Beispiel}
\newtheorem*{exas}{Beispiele}
\newtheorem*{eigen}{Eigenschaften}
\newtheorem*{rem}{Remark}
\newtheorem{remark}{Remark}
\newtheorem{qu}{Question}


\newcommand{\sub}{\underline}



\begin{abstract}
In this article we introduce the idea of a superimposed triangulation and clarify how these triangulations may be used to understand quiddity sequences for $\mathrm{SL}_3$-frieze patterns.
\end{abstract}

\maketitle
\tableofcontents

\definecolor{light-gray}{gray}{0.8}

\section{Introduction}

The theory of cluster algebras, introduced by Fomin and Zelevinsky in \cite{fz1}, \cite{fz2}, is remarkable for its ability to bring together many different areas of research. Perhaps the best illustration of this powerful theory is the frieze pattern. Frieze patterns convey the bijection between triangulations of a polygon and clusters of the homogeneous coordinate ring $\mathbb{C}[\mathrm{Gr}(2,n)]$, whilst remaining easy to understand. The bijection between triangulations of polygons and frieze patterns with positive integers was shown by Conway and Coxeter in \cite{cc1}, see also \cite{mg} for a survey, and the essential ingredient in this bijection is the quiddity sequence. 

Generalisations of frieze patterns have been studied in various forms for decades, starting with the $\mathrm{SL}_3$-frieze patterns introduced in \cite{cr}. However, quiddity sequences for the $\mathrm{SL}_3$-frieze patterns have remained elusive. 
We denote the set of clusters of Pl\"ucker coordinates in the cluster structure of $\mathbb{C}[\mathrm{Gr}(k,n)]$ by $\mathcal{S}_{k,n}$. The dimer algebras associated to each such cluster are interesting in their own right, and recent studies include \cite{bkm}, \cite{dl},  \cite{kul}, \cite{mc}, \cite{pas}, \cite{pr}; see also \cite{bo} and \cite{lam} for their connections to other areas of research. 

For particular clusters in $\mathcal{S}_{k,n}$, whose dimer algebras are \emph{rectangular} (see Definition \ref{rectify}), we can use results of Scott \cite{sco} and Oh-Postnikov-Speyer \cite{ops} to define \emph{$(k-1)$-superimposed triangulations of a polygon} (Definition \ref{super}) as well as \emph{higher quiddity sequences} (Definition \ref{dollazz}). 

\begin{prop}[ Proposition \ref{supertri}]
Let $\mathcal{C}\in\mathcal{S}_{k,n}$ be a rectangular cluster. Then $\mathcal{C}$ induces a pair of $2$-superimposed triangulations of an $n$-gon.
\end{prop}

This restricts to the following description, where a cluster of quadrilateral type refers to the quadrilateral arrangements found in \cite{sco}, see also Subsection \ref{cluck}. In this case we have a nice generalisation of the classical quiddity sequences introduced by Conway and Coxeter \cite{cc1}.

\begin{theorem}[Theorem \ref{quid}]
Let $\mathcal{C}\in\mathcal{S}_{3,n}$ be a rectangular cluster of quadrilateral type. Then 
the first (respectively final) non-trivial row of the associated $\mathrm{SL}_3$-frieze pattern is a quiddity sequence of order $n$ $(a_1,a_2,\cdots,a_n)$ (respectively $(b_1,b_2,\cdots,b_n)$) obtained by defining each $a_i$ (respectively $b_i$) to be one plus the number of arcs containing $i$ in one of the respective $2$-superimposed triangulations. 
\end{theorem}

In general we can obtain a quiddity sequence from either a row of the $\mathrm{SL}_3$-frieze pattern, or from the $2$-superimposed triangulation, but they may not be equal. It is also of note that rectangular-shaped quivers, albeit on a torus, have been studied in \cite{beil}. The description of a cluster of Pl\"ucker coordinates in the coordinate ring $\mathbb{C}[\mathrm{Gr}(3,n)]$ as superimposing triangulations on a polygon contrasts with the description that tilting modules of higher Auslander algebras of linearly oriented type $A$, that are summands of the cluster-tilting modules, correspond to triangulations of a cyclic polytope \cite{ot}.

\section{Frieze Patterns}
\subsection{Coxeter's Frieze Patterns}

\begin{defn}\label{friezedef}
In the sense of Coxeter, a \emph{frieze pattern} is an infinite array of numbers that satisfies the following conditions:
\begin{itemize}
\item The array has finitely many rows .
\item The top and bottom rows consist of only zeroes; the second and penultimate rows consist of only ones.
\item Consecutive rows are displayed with a shift, and every diamond $$\begin{tikzpicture}
\node(a) at (0,0){$a$};
\node(b) at (-0.5,-0.5){$b$};
\node(c) at (0.5,-0.5){$c$};
\node(d) at (0,-1){$d$};
\end{tikzpicture}$$
satisfies the \emph{unimodular rule}: $bc-ad=1$. 
\end{itemize}
\end{defn}
Note that we omit, by convention, the top and bottom rows of zeroes from the frieze pattern.
\begin{eg}
Examples of Coxeter's frieze patterns include:

$$\begin{tikzpicture}
\node(a) at (4,0){$1$};
\node(b) at (4.4,0){$1$};
\node(c) at (0.8,0){$1$};
\node(d) at (1.2,0){$1$};
\node(e) at (1.6,0){$1$};
\node(f) at (2,0){$1$};
\node(g) at (2.4,0){$1$};
\node(h) at (2.8,0){$1$};
\node(i) at (3.2,0){$1$};
\node(j) at (3.6,0){$1$};
\node(oa) at (4.2,0.4){$1$};
\node(ob) at (0.6,0.4){$3$};
\node(oc) at (1,0.4){$1$};
\node(od) at (1.4,0.4){$2$};
\node(oe) at (1.8,0.4){$2$};
\node(of) at (2.2,0.4){$1$};
\node(og) at (2.6,0.4){$3$};
\node(oh) at (3,0.4){$1$};
\node(oi) at (3.4,0.4){$2$};
\node(oj) at (3.8,0.4){$2$};
\node(pa) at (4,0.8){$1$};
\node(pb) at (0.4,0.8){$2$};
\node(pc) at (0.8,0.8){$2$};
\node(pd) at (1.2,0.8){$1$};
\node(pe) at (1.6,0.8){$3$};
\node(pf) at (2,0.8){$1$};
\node(pg) at (2.4,0.8){$2$};
\node(ph) at (2.8,0.8){$2$};
\node(pi) at (3.2,0.8){$1$};
\node(pj) at (3.6,0.8){$3$};
\node(qa) at (0.2,1.2){$1$};
\node(qb) at (0.6,1.2){$1$};
\node(qc) at (1,1.2){$1$};
\node(qd) at (1.4,1.2){$1$};
\node(qe) at (1.8,1.2){$1$};
\node(qf) at (2.2,1.2){$1$};
\node(qg) at (2.6,1.2){$1$};
\node(qh) at (3,1.2){$1$};
\node(qi) at (3.4,1.2){$1$};
\node(qj) at (3.8,1.2){$1$};
\end{tikzpicture}$$

$$\begin{tikzpicture}
\node(a) at (4,0){$1$};
\node(b) at (4.4,0){$1$};
\node(c) at (0.8,0){$1$};
\node(d) at (1.2,0){$1$};
\node(e) at (1.6,0){$1$};
\node(f) at (2,0){$1$};
\node(g) at (2.4,0){$1$};
\node(h) at (2.8,0){$1$};
\node(i) at (3.2,0){$1$};
\node(j) at (3.6,0){$1$};
\node(oa) at (4.2,0.4){$1$};
\node(ob) at (0.6,0.4){$3$};
\node(oc) at (1,0.4){$1$};
\node(od) at (1.4,0.4){$3$};
\node(oe) at (1.8,0.4){$1$};
\node(of) at (2.2,0.4){$3$};
\node(og) at (2.6,0.4){$1$};
\node(oh) at (3,0.4){$3$};
\node(oi) at (3.4,0.4){$1$};
\node(oj) at (3.8,0.4){$3$};
\node(pa) at (4,0.8){$2$};
\node(pb) at (0.4,0.8){$2$};
\node(pc) at (0.8,0.8){$2$};
\node(pd) at (1.2,0.8){$2$};
\node(pe) at (1.6,0.8){$2$};
\node(pf) at (2,0.8){$2$};
\node(pg) at (2.4,0.8){$2$};
\node(ph) at (2.8,0.8){$2$};
\node(pi) at (3.2,0.8){$2$};
\node(pj) at (3.6,0.8){$2$};
\node(qa) at (0.2,1.2){$3$};
\node(qb) at (0.6,1.2){$1$};
\node(qc) at (1,1.2){$3$};
\node(qd) at (1.4,1.2){$1$};
\node(qe) at (1.8,1.2){$3$};
\node(qf) at (2.2,1.2){$1$};
\node(qg) at (2.6,1.2){$3$};
\node(qh) at (3,1.2){$1$};
\node(qi) at (3.4,1.2){$3$};
\node(qj) at (3.8,1.2){$1$};
\node(ra) at (0,1.6){$1$};
\node(rb) at (0.4,1.6){$1$};
\node(rc) at (0.8,1.6){$1$};
\node(rd) at (1.2,1.6){$1$};
\node(re) at (1.6,1.6){$1$};
\node(rf) at (2,1.6){$1$};
\node(rg) at (2.4,1.6){$1$};
\node(rh) at (2.8,1.6){$1$};
\node(ri) at (3.2,1.6){$1$};
\node(rj) at (3.6,1.6){$1$};
\end{tikzpicture}$$

$$\begin{tikzpicture}
\node(a) at (4,0){$1$};
\node(b) at (4.4,0){$1$};
\node(c) at (0.8,0){$1$};
\node(d) at (1.2,0){$1$};
\node(e) at (1.6,0){$1$};
\node(f) at (2,0){$1$};
\node(g) at (2.4,0){$1$};
\node(h) at (2.8,0){$1$};
\node(i) at (3.2,0){$1$};
\node(j) at (3.6,0){$1$};
\node(oa) at (4.2,0.4){$2$};
\node(ob) at (0.6,0.4){$4$};
\node(oc) at (1,0.4){$1$};
\node(od) at (1.4,0.4){$2$};
\node(oe) at (1.8,0.4){$2$};
\node(of) at (2.2,0.4){$2$};
\node(og) at (2.6,0.4){$1$};
\node(oh) at (3,0.4){$4$};
\node(oi) at (3.4,0.4){$1$};
\node(oj) at (3.8,0.4){$2$};
\node(pa) at (4,0.8){$3$};
\node(pb) at (0.4,0.8){$3$};
\node(pc) at (0.8,0.8){$3$};
\node(pd) at (1.2,0.8){$1$};
\node(pe) at (1.6,0.8){$3$};
\node(pf) at (2,0.8){$3$};
\node(pg) at (2.4,0.8){$1$};
\node(ph) at (2.8,0.8){$3$};
\node(pi) at (3.2,0.8){$3$};
\node(pj) at (3.6,0.8){$1$};
\node(qa) at (0.2,1.2){$2$};
\node(qb) at (0.6,1.2){$2$};
\node(qc) at (1,1.2){$2$};
\node(qd) at (1.4,1.2){$1$};
\node(qe) at (1.8,1.2){$4$};
\node(qf) at (2.2,1.2){$1$};
\node(qg) at (2.6,1.2){$2$};
\node(qh) at (3,1.2){$2$};
\node(qi) at (3.4,1.2){$2$};
\node(qj) at (3.8,1.2){$1$};
\node(ra) at (0,1.6){$1$};
\node(rb) at (0.4,1.6){$1$};
\node(rc) at (0.8,1.6){$1$};
\node(rd) at (1.2,1.6){$1$};
\node(re) at (1.6,1.6){$1$};
\node(rf) at (2,1.6){$1$};
\node(rg) at (2.4,1.6){$1$};
\node(rh) at (2.8,1.6){$1$};
\node(ri) at (3.2,1.6){$1$};
\node(rj) at (3.6,1.6){$1$};
\end{tikzpicture}$$
\end{eg}

A frieze pattern is said to be of width $n$ if it has $n$ rows strictly between the border rows of ones at the top and bottom.

\subsection{$\mathrm{SL}_3$-Frieze Patterns}\label{slfree}

The easiest way to understand the unimodular rule is to interpret it as a $2\times 2$-matrix determinant. So it makes sense to question whether there are frieze patterns defined by $3\times 3$-matrix determinants. As introduced in \cite{cr}, an \emph{$\mathrm{SL}_3$-frieze pattern} is defined as in Definition \ref{friezedef}, with the only changes being an extra row of zeroes at each of the top and bottom of the frieze, and that each $3\times 3$-matrix has determinant one. Examples are given in Example \ref{fig1ii}. By convention, we again neglect to write the zeroes in the pattern. 

\begin{eg}
Examples of $\mathrm{SL}_3$-friezes include\label{fig1ii}

$$\begin{tikzpicture}
\node(a) at (4,0){$1$};
\node(b) at (4.4,0){$1$};
\node(c) at (0.8,0){$1$};
\node(d) at (1.2,0){$1$};
\node(e) at (1.6,0){$1$};
\node(f) at (2,0){$1$};
\node(g) at (2.4,0){$1$};
\node(h) at (2.8,0){$1$};
\node(i) at (3.2,0){$1$};
\node(j) at (3.6,0){$1$};
\node(oa) at (4.2,0.4){$2$};
\node(ob) at (0.6,0.4){$2$};
\node(oc) at (1,0.4){$4$};
\node(od) at (1.4,0.4){$1$};
\node(oe) at (1.8,0.4){$2$};
\node(of) at (2.2,0.4){$4$};
\node(og) at (2.6,0.4){$1$};
\node(oh) at (3,0.4){$2$};
\node(oi) at (3.4,0.4){$4$};
\node(oj) at (3.8,0.4){$1$};
\node(pa) at (4,0.8){$1$};
\node(pb) at (0.4,0.8){$1$};
\node(pc) at (0.8,0.8){$4$};
\node(pd) at (1.2,0.8){$2$};
\node(pe) at (1.6,0.8){$1$};
\node(pf) at (2,0.8){$4$};
\node(pg) at (2.4,0.8){$2$};
\node(ph) at (2.8,0.8){$1$};
\node(pi) at (3.2,0.8){$4$};
\node(pj) at (3.6,0.8){$2$};
\node(qa) at (0.2,1.2){$1$};
\node(qb) at (0.6,1.2){$1$};
\node(qc) at (1,1.2){$1$};
\node(qd) at (1.4,1.2){$1$};
\node(qe) at (1.8,1.2){$1$};
\node(qf) at (2.2,1.2){$1$};
\node(qg) at (2.6,1.2){$1$};
\node(qh) at (3,1.2){$1$};
\node(qi) at (3.4,1.2){$1$};
\node(qj) at (3.8,1.2){$1$};
\end{tikzpicture}$$

$$\begin{tikzpicture}
\node(a) at (4,0){$1$};
\node(b) at (4.4,0){$1$};
\node(c) at (0.8,0){$1$};
\node(d) at (1.2,0){$1$};
\node(e) at (1.6,0){$1$};
\node(f) at (2,0){$1$};
\node(g) at (2.4,0){$1$};
\node(h) at (2.8,0){$1$};
\node(i) at (3.2,0){$1$};
\node(j) at (3.6,0){$1$};
\node(oa) at (4.2,0.4){$1$};
\node(ob) at (0.6,0.4){$2$};
\node(oc) at (1,0.4){$3$};
\node(od) at (1.4,0.4){$2$};
\node(oe) at (1.8,0.4){$1$};
\node(of) at (2.2,0.4){$6$};
\node(og) at (2.6,0.4){$1$};
\node(oh) at (3,0.4){$2$};
\node(oi) at (3.4,0.4){$3$};
\node(oj) at (3.8,0.4){$2$};
\node(pa) at (4,0.8){$1$};
\node(pb) at (0.4,0.8){$1$};
\node(pc) at (0.8,0.8){$3$};
\node(pd) at (1.2,0.8){$3$};
\node(pe) at (1.6,0.8){$1$};
\node(pf) at (2,0.8){$3$};
\node(pg) at (2.4,0.8){$3$};
\node(ph) at (2.8,0.8){$1$};
\node(pi) at (3.2,0.8){$3$};
\node(pj) at (3.6,0.8){$3$};
\node(qa) at (0.2,1.2){$1$};
\node(qb) at (0.6,1.2){$1$};
\node(qc) at (1,1.2){$1$};
\node(qd) at (1.4,1.2){$1$};
\node(qe) at (1.8,1.2){$1$};
\node(qf) at (2.2,1.2){$1$};
\node(qg) at (2.6,1.2){$1$};
\node(qh) at (3,1.2){$1$};
\node(qi) at (3.4,1.2){$1$};
\node(qj) at (3.8,1.2){$1$};
\end{tikzpicture}$$
\end{eg}

\section{Background}
\subsection{Pl\"ucker Relations}
Recall that the Grassmannian of all $k$-dimensional subspaces of $\mathbb{C}^n$, $\mathrm{Gr}(k,n)$, can be embedded into the projective space $\mathbb{P}(\wedge^k(\mathbb{C}^n))$ via the Pl\"ucker embedding. The coordinates of $\wedge^k(\mathbb{C}^n)$ are called the \emph{Pl\"ucker coordinates} and are indexed by the $k$-subsets $I=\{i_1,i_2,\cdots,i_k\}\subset\{1,\cdots ,n\},$ where $1\leq i_1\leq i_2\leq \cdots\leq i_k\leq n$. The coordinate defined by an ordered subset $\{i_1,i_2,\cdots,i_k\}$ will be denoted $p_{i_1i_2\cdots i_k}$. 

The ordering of elements in a $k$-multiset may not always be known; the definition of the Pl\"ucker coordinates may be extended to allow for this. By convention, rearrangements of Pl\"ucker coordinates are treated by setting 
\begin{align} p_{i_1\cdots i_ri_s\cdots i_k}=-p_{i_1\cdots i_si_r\cdots i_k}.\end{align}
 In particular, if $i_r=i_s$ then $p_{i_1\cdots i_ri_r\cdots i_k}=0.$
The Pl\"ucker embedding satisfies the determinantal \emph{Pl\"ucker relations}. In the case $k=2$, the Pl\"ucker relations are 
$$p_{ac}p_{bd}=p_{ab}p_{cd}+p_{ad}p_{bc},$$
for all $1\leq a<b<c<d\leq n$. For a $(k-1)$-multiset $I=\{i_1,\cdots,i_{k-1}\}\subset \{1,2,\cdots,n\}$, let $p_{Ii_{k}}$ be the Pl\"ucker coordinate for the $k$-multiset $\{i_1,\cdots,i_{k-1},i_k\}$; $p_{Ii_{k}i_{k+1}}$  the Pl\"ucker coordinate for the $k+1$-multiset $\{i_1,\cdots,i_{k-1},i_k, i_{k+1}\}$ and so on.
Then, more generally, the Pl\"ucker relations for $k>2$ are generated by the relations \begin{align}
p_{Iac}p_{Ibd}=p_{Iab}p_{Icd}+p_{Iad}p_{Ibc},
\end{align} where $I$ is a $(k-2)$-muliset of $\{1,\cdots, n\}$. Another set of generating relations for the $(k,n)$-Pl\"ucker relations is the set of relations

\begin{align}
\sum_{r=0}^k(-1)^np_{i_1i_2\cdots i_{k-1}j_r}p_{j_0\cdots \widehat{j_r}\cdots j_k}=0
\end{align}
 where $1\leq i_0<i_1<\cdots<i_{k-1}\leq n$ and $1\leq j_0< j_1< \cdots j_k\leq n$.
\subsection{Alternating Strand Diagrams}

Scott's work on Grassmannian cluster algebras in \cite{sco} is vital to the understanding of $\mathrm{SL}_k$-frieze patterns. Underlying the proof that the homogeneous coordinate ring of a Grassmannian $\mathrm{Gr}(k,n)$ is a cluster algebra are objects called alternating strand diagrams. 

\begin{defn}

Two $k$-subsets $I$ and $J$ are said to be \emph{non-crossing} (also referred to as \emph{weakly separated} in some articles such as \cite{ops}
) if there do not exist elements $s<t<u<v$ (ordered modulo $n$) where $s,u\in I$, $s,u\notin J$ and $t,v\in J$, $t,v\notin I$. A \emph{cluster of Pl\"ucker coordinates in the cluster structure of $\mathrm{Gr}(k,n)$} is a maximal collection of pairwise non-crossing $k$-subsets of $\{1,2,\cdots,n\}$. It was proven in \cite{dkk}, \cite{ops} that a collection of pairwise non-crossing $k$-subsets of $\{1,2,\cdots,n\}$ is maximal if and only if it has $(k-1)(n-k-1)+n$ members. 
\end{defn}
Note that there is bijection between maximal collections of pairwise non-crossing $k$-subsets of $\{1,2,\cdots,n\}$ and maximal collections of pairwise non-crossing $(n-k)$-subsets of $\{1,2,\cdots,n\}$ obtained by replacing each subset in the collection with its complement. 

Let $C$ be a disc with $n$ vertices $\{1,2,\cdots, n\}=:C_1$ prescribed clockwise around its boundary circle. A {\em $(k,n)$-alternating strand diagram} $D$ consists of $n$ directed curves called {\em strands} between the points of $C_1$, such that at each vertex $i\in\left\{1,\cdots, n\right\}$ there is a single strand starting at $i$, which then ends at the vertex $i+k\ (\mathrm{mod}\ n)$, subject to the conditions:

\begin{itemize}
\item Only two strands may cross at a given point, and all crossings are transversal.
\item There are finitely many crossing points.
\item Proceeding along a given strand, the other strands crossing it alternate between crossing it left to right and right to left.
\item A strand cannot intersect itself.
\item If two strands intersect at points $x$ and $y$ (where $x$ an $y$ are also allowed to be points along the boundary), then one strand is oriented from $x$ to $y$ and the other from $y$ to $x$.
\end{itemize}

If in addition, upon proceeding along a given strand, no other strand crosses it twice in a row, then the alternating strand diagram is said to be \emph{reduced}.
Note that alternating strand diagrams must be considered up to isotopy: two alternating strand diagrams are said to be equivalent if they can be obtained from each other via ``twists". Any alternating stand diagram can be obtained from one of reduced type via twists, and henceforth we will assume all alternating strand diagrams are of reduced type. For more details on twists and dimer algebras arising from alternating strand diagrams, see \cite{bkm}. 

\begin{eg}\label{asd}
The following is an example of a (reduced) $(2,5)$-alternating strand diagram, $D$.

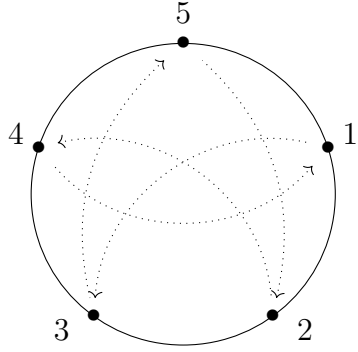
$$
\begin{tikzpicture}[scale=2]
\draw (0,0) circle(1cm);
\node(A) at (0.951,0.309){$\bullet$};
\node(B) at (0,1){$\bullet$};
\node(C) at (-0.951,0.309){$\bullet$};
\node(D) at (-0.5878,-0.809){$\bullet$};
\node(E) at (0.5878,-0.809){$\bullet$};
\node(a) at (1.1,0.4){$1$};
\node(b) at (0,1.2){$5$};
\node(c) at (-1.1,0.4){$4$};
\node(d) at (-0.8,-0.9){$3$};
\node(e) at (0.8,-0.9){$2$};
\path[<-,dotted,bend left=50]
(A) edgenode[above]{} (C)
(C) edgenode[above]{} (E);
\path[<-,dotted,bend right=30]
(E) edgenode[above]{} (B)
(B) edgenode[above]{} (D);
\path[<-,dotted,bend left=50]
(D) edgenode[above]{} (A);
\end{tikzpicture}$$

\end{eg}

Recall that a quiver $Q$ is a directed graph with a set $Q_0$ of vertices and a set $Q_1$ of arrows. 

\begin{defn}\label{label}
Let $D$ be a $(k,n)$-alternating strand diagram. The corresponding {\em dimer algebra associated with $D$}, $A(D):=kQ(D)/I(D)$ is defined as follows. The quiver $Q(D)$ has a vertex for each region of $D$ that is bounded by strands which have alternating orientation, and arrows corresponding to the intersection points of two alternating regions. As we are assuming that $D$ is a reduced $(k,n)$-alternating strand diagram, there can be no two-cycles in the quiver $Q(D)$. There is a set of boundary arrows which connect the vertices on the boundary of the disc, with orientation depending on the order of the two incident strands. Each vertex has a \emph{label} given by a $k$-subset $I$ of $\{1,2,\cdots,n\}$; each element $i\in I$  given by a strand from $i$ to $i+k$ that contains the given region on its left. By convention, each label $\{i_1,i_2,\cdots i_k\}$ will be displayed as $i_1i_2\cdots i_k$. 

Every non-boundary arrow $\alpha$ in $Q(D)$ is contained in two cycles, one clockwise and one anti-clockwise. These cycles are $\alpha\rho_\alpha^+$ and $\alpha\rho_\alpha^-$ respectively. Set $$I(D)=\langle \rho_\alpha^+-\rho_\alpha^-|\alpha \text{ a non-boundary arrow in $Q(D)$}\rangle.$$ 
\end{defn}

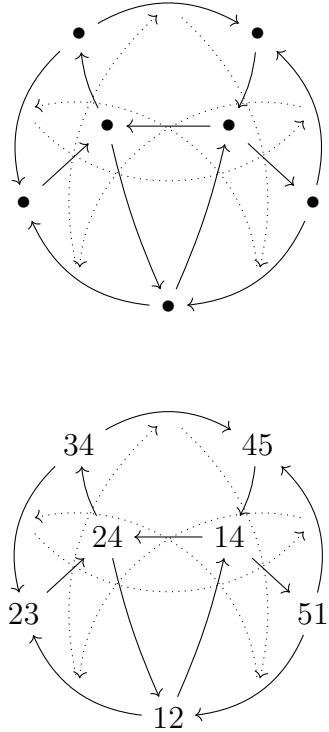
\begin{figure}
\caption{For the (2,5)-alternating strand diagram $D$ in Example \ref{asd}, the quiver $Q(D)$ is as follows. Below, we add labels to each vertex.}\label{asde}
$$\begin{tikzpicture} 
[scale=2]
\node(A) at (0.951,0.309){};
\node(B) at (0,1){};
\node(C) at (-0.951,0.309){};
\node(D) at (-0.5878,-0.809){};
\node(E) at (0.5878,-0.809){};
\node(F) at (0.5878,0.809){$\bullet$};
\node(G) at (-0.5878,0.809){$\bullet$};
\node(H) at (-0.9511,-0.3090){$\bullet$};
\node(I) at (0,-1){$\bullet$};
\node(J) at (0.9511,-0.3090){$\bullet$};
\node(K) at (-0.4,0.2){$\bullet$};
\node(L) at (0.4,0.2){$\bullet$};
\path[<-,dotted,bend left=50]
(A) edgenode[above,]{} (C)
(C) edgenode[above]{} (E);
\path[<-,dotted,bend right=30]
(E) edgenode[above]{} (B)
(B) edgenode[above]{} (D);
\path[<-,dotted,bend left=50]
(D) edgenode[above]{} (A);
\path[<-,bend right=10] (L)edgenode[above]{} (F);
\path[<-] (K)edgenode[above]{} (L);
\path[<-, bend left=5] (L)edgenode[above]{} (I);
\path[<-] (J)edgenode[above]{} (L);
\path[<-] (K)edgenode[above]{} (H);
\path[<-,bend right=10] (G)edgenode[above]{} (K);
\path[<-,bend left=5] (I)edgenode[above]{} (K);
\path[<-,bend right=30]
(F) edgenode[above]{} (G);
\path[<-,bend left=30]
(H) edgenode[above]{} (G);
\path[<-,bend right=30]
(H) edgenode[above]{} (I);
\path[<-,bend right=30]
(I) edgenode[above]{} (J);
\path[<-,bend left=30]
(F) edgenode[above]{} (J);
\end{tikzpicture}$$

$$\begin{tikzpicture} 
[scale=2]
\node(A) at (0.951,0.309){};
\node(B) at (0,1){};
\node(C) at (-0.951,0.309){};
\node(D) at (-0.5878,-0.809){};
\node(E) at (0.5878,-0.809){};
\node(F) at (0.5878,0.809){$45$};
\node(G) at (-0.5878,0.809){$34$};
\node(H) at (-0.9511,-0.3090){$23$};
\node(I) at (0,-1){$12$};
\node(J) at (0.9511,-0.3090){$51$};
\node(K) at (-0.4,0.2){$24$};
\node(L) at (0.4,0.2){$14$};
\path[<-,dotted,bend left=50]
(A) edgenode[above,]{} (C)
(C) edgenode[above]{} (E);
\path[<-,dotted,bend right=30]
(E) edgenode[above]{} (B)
(B) edgenode[above]{} (D);
\path[<-,dotted,bend left=50]
(D) edgenode[above]{} (A);
\path[<-,bend right=10] (L)edgenode[above]{} (F);
\path[<-] (K)edgenode[above]{} (L);
\path[<-, bend left=5] (L)edgenode[above]{} (I);
\path[<-] (J)edgenode[above]{} (L);
\path[<-] (K)edgenode[above]{} (H);
\path[<-,bend right=10] (G)edgenode[above]{} (K);
\path[<-,bend left=5] (I)edgenode[above]{} (K);
\path[<-,bend right=30]
(F) edgenode[above]{} (G);
\path[<-,bend left=30]
(H) edgenode[above]{} (G);
\path[<-,bend right=30]
(H) edgenode[above]{} (I);
\path[<-,bend right=30]
(I) edgenode[above]{} (J);
\path[<-,bend left=30]
(F) edgenode[above]{} (J);
\end{tikzpicture}$$
\end{figure}

\begin{theorem}\label{scott} \cite{sco}
There is a bijection between the clusters in $\mathcal{S}_{k,n}$ and (reduced) $(k,n)$-alternating strand diagrams. 
\end{theorem}

\subsection{Plabic Graphs}

Plabic graphs were introduced in \cite{post}, see also \cite{bkm}, \cite{ops}. Just as in Defintion \ref{label}, where a $k$-subset was associated to each alternating region, we may associate a colour to each oriented region: specifically, we draw a white node in each clockwise region and a black node in each anti-clockwise region, as in Figure \ref{plab}. In addition, a (non-coloured) boundary vertex is added to each of the initial marked points on the disc. The nodes of neighbouring oriented regions are then connected with an edge to form a \emph{plabic graph}. Note that, just as alternating strand diagrams may be reduced, there are reduced plabic graphs - a plabic graph is \emph{reduced} if each black node neighbours only white nodes and vice versa. A reduced alternating strand diagram gives rise to a reduced plabic graph. 

Fix a cluster $\mathcal{C}$ in $\mathcal{S}_{k,n}$. By Theorem \ref{scott}, $\mathcal{C}$ has an associated alternating strand diagram, and hence we can define the (reduced) \emph{plabic graph associated with $\mathcal{C}$}. 
\begin{figure}
\caption{The plabic graph associated to the alternating strand diagram $D$ from Example \ref{asd}, where we have omitted the boundary vertices.}\label{plab}
$$\begin{tikzpicture} 
[scale=2]
\node(A) at (0.951,0.309){};
\node(B) at (0,1){};
\node(C) at (-0.951,0.309){};
\node(D) at (-0.5878,-0.809){};
\node(E) at (0.5878,-0.809){};
\node(F) at (0.5878,0.809){$45$};
\node(G) at (-0.5878,0.809){$34$};
\node(H) at (-0.9511,-0.3090){$23$};
\node(I) at (0,-1){$12$};
\node(J) at (0.9511,-0.3090){$51$};
\node(K) at (-0.4,0.2){$24$};
\node(L) at (0.4,0.2){$14$};
\node(M) at (0,0.505){$\circ$};
\node(N) at (-0.43,-0.37){$\circ$};
\node(O) at (0.43,-0.37){$\circ$};
\node(P) at (-0.6463,0.233){$\bullet$};
\node(Q) at (0.6463,0.233){$\bullet$};
\node(R) at (0,-0.2){$\bullet$};

\path[<-,bend right=10] (L)edgenode[above]{} (F);
\path[<-] (K)edgenode[above]{} (L);
\path[<-, bend left=5] (L)edgenode[above]{} (I);
\path[<-] (J)edgenode[above]{} (L);
\path[<-] (K)edgenode[above]{} (H);
\path[<-,bend right=10] (G)edgenode[above]{} (K);
\path[<-,bend left=5] (I)edgenode[above]{} (K);
\path[<-,bend right=30]
(F) edgenode[above]{} (G);
\path[<-,bend left=30]
(H) edgenode[above]{} (G);
\path[<-,bend right=30]
(H) edgenode[above]{} (I);
\path[<-,bend right=30]
(I) edgenode[above]{} (J);
\path[<-,bend left=30]
(F) edgenode[above]{} (J);

\path[dotted] (R) edge (M);
\path[dotted] (R) edge (N);
\path[dotted] (R) edge (O);
\path[dotted] (P) edge (M);
\path[dotted] (Q) edge (M);
\path[dotted] (P) edge (N);
\path[dotted] (Q) edge (O);

\end{tikzpicture}$$
\end{figure}
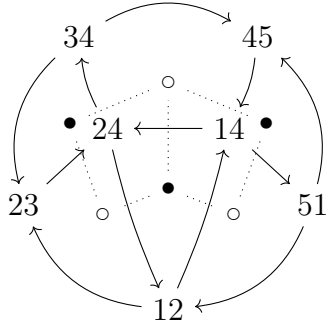

Let $\mathcal{C}$ be a cluster in $\mathcal{S}_{k,n}$, and $K$ be a $(k-1)$-subset of $\{1,2,\cdots, n\}$. If there are at least three subsets $I\in \mathcal{C}$ that satisfy $K\subset I$, then the collection of $I\in \mathcal{C}$ such that $K\subset I$ is a \emph{white clique in $\mathcal{C}$}, $\mathcal{W}(K)$. Each white clique is of the form $$(K\cup\{i_0\},K\cup\{i_1\},\cdots, K\cup\{i_l\}),$$ 
where $i_0<i_1<\cdots <i_l$ modulo $n$.

Likewise, let $L$ be a $(k+1)$-subset of $\{1,2,\cdots,n\}$. If there are at least three subsets $J\in\mathcal{C}$ such that $J\subset L$, then the collection of $J\in \mathcal{C}$ such that $J\subset L$ defines a \emph{black clique in $\mathcal{C}$}, $\mathcal{B}(L)$ is. Each black clique is of the form $$(L\setminus\{j_0\},L\setminus\{j_1\},\cdots, L\setminus\{j_m\}),$$ 
where $j_0<j_1<\cdots <j_m$ modulo $n$.  

The white (respectively black) cliques in $\mathcal{C}$ correspond to the white (respectively black) vertices in the plabic graph associated with $\mathcal{C}$. In Figure \ref{plab}, there are three white cliques: $$(41,42,43,45),$$ $$(12,14,15),$$ $$(21,23,24),$$ which are $\mathcal{W}(\{4\})$, $\mathcal{W}(\{1\})$ and $\mathcal{W}(\{2\})$ respectively. Likewise, there are three black cliques: $$(45,15,14),$$ $$(24,14,12),$$ $$(34,24,23),$$ which are the black cliques $\mathcal{B}(\{1,4,5\})$, $\mathcal{B}(\{1,2,4\})$ and $\mathcal{B}(\{2,3,4\})$. 

\subsection{Mutation}\label{muted}

 To summarise the previous sections, each cluster $\mathcal{C}\in\mathcal{S}_{k,n}$ can be described as a maximal collection of pairwise non-crossing $k$-subsets of $\{1,2,\cdots,n\}$. To each cluster we may associate a (reduced) $(k,n)$-alternating strand diagram, a dimer algebra and a (reduced) plabic graph. 

\begin{defn}\label{mute}
Let $\mathcal{C} \in \mathcal{S}_{k,n}$ and $I\cup \{a,c\}=:J\in\mathcal{C}$ for some $(k-2)$-subset $I$. Then if there exist $b,d\in\{1,2,\cdots,n\}$ such that $a<b<c<d$ modulo $n$ and the $k$-subsets $I\cup \{a,b\},I\cup \{c,d\},I\cup \{a,d\},I\cup \{b,c\}\in\mathcal{C}$, then $\mathcal{C}$ may be \emph{mutated at $I\cup \{a,c\}$} by replacing $I\cup \{a,c\}$ with the $k$-subset $I\cup \{b,d\}$.
\end{defn}

For a given cluster $\mathcal{C} \in \mathcal{S}_{k,n}$, a mutation at $J\in\mathcal{C}$ to a new cluster $\mathcal{C}^\prime\in\mathcal{S}_{k,n}$ is only possible when there are four subsets as in Definition \ref{mute}. Equivalently, a mutation at a non-interval subset $J$ is possible if $J$ is contained in precisely two white cliques and two black cliques in $\mathcal{C}$. Alternatively, a mutation at  a non-interval subset $J$ is possible if the vertex $j$ corresponding to $J$, in the quiver of the dimer algebra associated with $\mathcal{C}$ has precisely two arrows starting and two arrows ending at it. In this case, the effect on the quiver $Q$ of the dimer algebra associated with $\mathcal{C}$ can be described using Fomin-Zelevinsky quiver mutation \cite{fz1}:

\begin{enumerate}
\item For all possible paths $i\rightarrow j \rightarrow k$ in $Q$, add an arrow $i\rightarrow k$.
\item Reverse all arrows incident with $j$.
\item Remove all possible two-cycles.
\end{enumerate}

For more information on mutations of clusters in $\mathcal{S}_{k,n}$, we refer to \cite{bkm} for mutations of dimer algebras, \cite{ops} for mutations of plabic graphs and \cite{sco} for mutations of alternating strand diagrams.

\section{Rectangular frieze patterns}

\begin{defn}\label{rectify}
Following Scott \cite[Section 4]{sco}, a $k$-subset $I\subset\{1,2,\cdots,n\}$ is a \emph{double-interval subset} if it is of the form $I=I_1\bigsqcup I_2$ where $I_1$ and $I_2$ are interval subsets. 
A cluster $\mathcal{C} \in\mathcal{S}_{k,n}$ is \emph{rectangular} if every subset in $\mathcal{C}$ is a double-interval subset. 
\end{defn}

Every interval subset is automatically a double-interval subset. We will, again following the notation of Scott, sometimes partition the label of a double-interval subset with a comma to show the double interval. For example, we may write the label of the 4-subset $\{1,2,3,7\}$ as $123,7$.

\begin{figure}
\caption{An example of the quiver and plabic graph associated to a rectangular cluster in $\mathcal{S}_{k,n}$, where again we have omitted the boundary vertices of the plabic graph. }\label{wrecking}
$$\begin{tikzpicture}[scale=1]
\node(a) at (1,1){7891};
\node(e) at (11,1){5678};
\node(ai) at (4,1){6789};
\node(dd) at (3,3){2,789};
\node(df) at (5,3){2,678};
\node(dh) at (7,3){3,678};
\node(dj) at (9,3){4,678};
\node(em) at (11,4){4567};
\node(fd) at (3,5){12,78};
\node(ff) at (5,5){23,78};
\node(fh) at (7,5){23,67};
\node(fj) at (9,5){34,67};
\node(hd) at (3,7){912,7};
\node(hf) at (5,7){123,7};
\node(hh) at (7,7){234,7};
\node(hj) at (9,7){345,7};
\node(ia) at (1,9){8912};
\node(ig) at (6,9){1234};
\node(ik) at (11,9){3456};
\node(ie) at (4,9){9123};
\node(ii) at (8,9){2345};

\node(f) at (8/3,25/3){\large{$\circ$}};
\node(g) at (5,25/3){\large{$\circ$}};
\node(h) at (7,25/3){\large{$\circ$}};
\node(i) at (28/3,25/3){\large{$\circ$}};
\node(j) at (11/3,19/3){\large{$\circ$}};
\node(k) at (6,6){\large{$\circ$}};
\node(l) at (25/3,19/3){\large{$\circ$}};
\node(m) at (4,4){\large{$\circ$}};
\node(n) at (23/3,13/3){\large{$\circ$}};
\node(o) at (8/3,5/3){\large{$\circ$}};
\node(p) at (36/5,1.75){\large{$\circ$}};
\node(q) at (29/3,4){\large{$\circ$}};

\node(mf) at (4,23/3){\large{$\bullet$}};
\node(mg) at (6,23/3){\large{$\bullet$}};
\node(mh) at (8,23/3){\large{$\bullet$}};
\node(mi) at (11/5,5){\large{$\bullet$}};
\node(mj) at (13/3,17/3){\large{$\bullet$}};
\node(mk) at (23/3,17/3){\large{$\bullet$}};
\node(ml) at (10,25/4){\large{$\bullet$}};
\node(mm) at (6,4){\large{$\bullet$}};
\node(mn) at (25/3,11/3){\large{$\bullet$}};
\node(mo) at (4,7/3){\large{$\bullet$}};
\node(mp) at (31/3,8/3){\large{$\bullet$}};

\path[-,dotted] (f) edge (mf);
\path[-,dotted] (mf) edge (g);
\path[-,dotted] (g) edge (mg);
\path[-,dotted] (mg) edge (h);
\path[-,dotted] (h) edge (mh);
\path[-,dotted] (mh) edge (i);

\path[-,dotted] (mi) edge (j);
\path[-,dotted] (j) edge (mj);
\path[-,dotted] (mj) edge (k);
\path[-,dotted] (k) edge (mk);
\path[-,dotted] (mk) edge (l);
\path[-,dotted] (l) edge (ml);

\path[-,dotted] (mi) edge (m);
\path[-,dotted] (m) edge (mm);
\path[-,dotted] (mm) edge (n);
\path[-,dotted] (n) edge (mn);
\path[-,dotted] (mn) edge (q);

\path[-,dotted] (o) edge (mo);
\path[-,dotted] (mo) edge (p);
\path[-,dotted] (p) edge (mp);

\path[-,dotted] (f) edge (mi);
\path[-,dotted] (mi) edge (o);

\path[-,dotted] (mf) edge (j);
\path[-,dotted] (mj) edge (m);
\path[-,dotted] (mo) edge (m);

\path[-,dotted] (mg) edge (k);
\path[-,dotted] (k) edge (mm);
\path[-,dotted] (p) edge (mm);

\path[-,dotted] (mh) edge (l);
\path[-,dotted] (mk) edge (n);
\path[-,dotted] (mn) edge (p);

\path[-,dotted] (i) edge (ml);
\path[-,dotted] (ml) edge (q);
\path[-,dotted] (q) edge (mp);

\path[->] (df) edge (dd);
\path[->] (df) edge (dh);
\path[->] (dh) edge (dj);
\path[->] (fd) edge (ff);
\path[->] (fh) edge (ff);
\path[->] (fh) edge (fj);
\path[->] (hd) edge (hf);
\path[->] (hf) edge (hh);
\path[->] (hh) edge (hj);
\path[->] (dd) edge (fd);
\path[->] (fd) edge (hd);
\path[->] (ff) edge (df);
\path[->] (ff) edge (hf);
\path[->] (dh) edge (fh);
\path[->] (hh) edge (fh);
\path[->] (dj) edge (fj);
\path[->] (hj) edge (fj);
\path[->] (ie) edge (hd);
\path[->] (hf) edge (ie);
\path[->] (ii) edge (hh);
\path[->] (hj) edge (ii);
\path[->] (ik) edge (hj);
\path[->] (fj) edge (em);
\path[->] (em) edge (dj);
\path[->] (ia) edge (ie);
\path[->] (ie) edge (ig);
\path[->] (ig) edge (ii);
\path[->] (ii) edge (ik);
\path[->] (em) edge (ik);
\path[->] (e) edge (em);
\path[->] (e) edge (ai);
\path[->] (ai) edge (a);
\path[->] (ia) edge (a);

\path[->] (a) edge (dd);
\path[->] (dd) edge (ai);
\path[->] (ai) edge (df);
\path[->] (dj) edge (e);
\path[->] (hd) edge (ia);
\path[->] (hh) edge (ig);
\path[->] (ig) edge (hf);

\path[->] (fj) edge (dh);
\path[->] (hf) edge (fd);
\path[->] (fj) edge (hh);

\end{tikzpicture}$$

\end{figure}
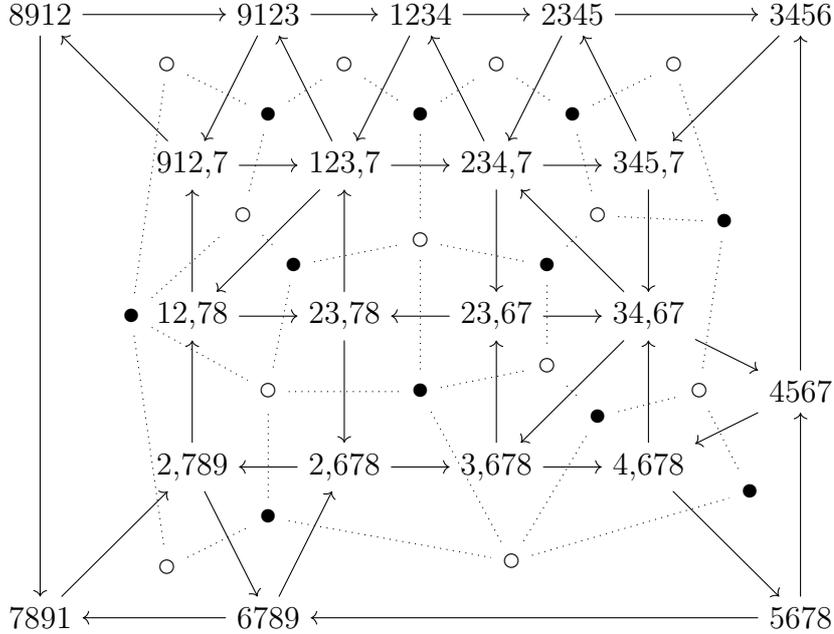

\subsection{Cliques for Rectangular Clusters}\label{cluck}

Let $\mathcal{C}$ be a rectangular cluster in the cluster structure of $\mathrm{Gr}(k,n)$. We say that a non-interval subset $I\in\mathcal{C}$ is an \emph{horizontal edge subset of $\mathcal{C}$} if $I=\{i,i+1,\cdots, i+k-2,j\}$, a \emph{vertical edge subset of $\mathcal{C}$} if $I=\{i,i+1,\cdots \widehat{j},\cdots, i+k\}$. We say that $I$ is an \emph{edge subset of $\mathcal{C}$} if it is either a horizontal or a vertical edge subset of $\mathcal{C}$ and $I$ is a \emph{corner subset of $\mathcal{C}$} if it is both. It should be noted that if we consider $\mathcal{C}$ as a cluster in $\mathcal{S}_{n-k,n}$ by replacing each $k$-subset in $\mathcal{C}$ with its complement, then vertical edge subsets become horizontal edge subsets and vice versa. 

For the rectangular cluster $\mathcal{C}$ depicted in Figure \ref{wrecking}, the 4-subsets with labels $1237$, $2347$, $2678$ and $3678$ are all horizontal edge subsets of $\mathcal{C}$. Likewise the 4-subsets with labels $1278$ and $3467$ are vertical edge subsets of $\mathcal{C}$ and the 4-subsets with labels $2789$, $9127$, $3457$ and $4678$ are all corner subsets of $\mathcal{C}$.
 
Recall that for a cluster $\mathcal{C}\in \mathcal{S}_{k,n}$, a collection of $k$-subsets in $\mathcal{C}$ are in a white clique $\mathcal{W}(K)$ if they share a $(k-1)$-subset $K$, and in a black clique $\mathcal{B}(L)$ if they are all contained in the same $(k+1)$-subset $L$.

Two subsets $I,J\in \mathcal{C}$ are said to be \emph{neighbours} if their corresponding vertices in the quiver of the dimer algebra associated with $\mathcal{C}$ are connected by an arrow. 
In the quiver of the dimer algebra associated with $\mathcal{C}$, every vertex in the interior must have an even number of neighbours - each arrow must be contained in both a clockwise and an anti-clockwise cycle. In addition, we are assuming that this quiver contains no two-cycles. 
So each non-interval subset in $\mathcal{C}$ must have at least four neighbours. This implies that each non-interval subset in $\mathcal{C}$ must belong to at least two white cliques and two black cliques. 
 
\begin{lm}\label{clique}
Let $\mathcal{C}$ be a rectangular cluster in $\mathcal{S}_{k,n}$. Then the following hold:
\begin{itemize}
\item For each clique of $\mathcal{C}$ that contains an interval subset as a member, each non-interval subset in the clique must be an edge subset, and all edge subsets in this clique must be of the same type (horizontal or vertical).
\item Each clique of $\mathcal{C}$ that does not contain an interval subset must have either three or four members. 
\item If two non-interval subsets $I,J\in\mathcal{C}$ are neighbours, then there is a black clique $\mathcal{B}(I\cup J)$ and a white clique $\mathcal{W}(I\cap J)$ in $\mathcal{C}$.  \label{cliky}
\end{itemize}
\end{lm}

  \begin{proof}
Let $J$ be a $(k-1)$-subset defining a white clique in $\mathcal{C}$ that contains an interval subset. Then either $J=\{i,i+1,\cdots, i+k-2\}$ or $J=\{i,i+1,\cdots, \widehat{j},\cdots, i+k-1\}$.  Any member of the white clique defined by $J$ must be of the form $J\cup\{a\}$ for some $1\leq a\leq n$, $a\notin J$. In the first case, $J\cup\{a\}=\{i,i+1,\cdots, i+k-2,a\}$ is either a horizontal edge subset or an interval subset. 

In the second case, for $J\cup\{a\}$ to be a double-interval subset, we must have either $a=j$ and $J\cup\{a\}$ is an interval subset, or $a\in\{i-1,i+k\}$ in which case $J\cup\{a\}$ is a vertical edge subset. The proof for black cliques is similar.

For the second part, by definition, any clique in a reduced plabic graph must have at least three members. We will first show that any white clique defined by a non-interval $(k-1)$-subset $J$ has at most four members. Without loss of generality, let $$J=\{i,i+1,\cdots,i+l,j,j+1,\cdots,j+(k-l-3)\}.$$
There are only four possible double-interval subsets that may belong to the white clique defined by $J$: $J\cup\{i-1\}$, $J\cup\{i+l+1\}$, $J\cup\{j-1\}$ and $J\cup\{j+(k-l-2)\}$.
On the other hand, let $J$ be a non-interval $(k+1)$-subset defining a black clique in $\mathcal{C}$. So without loss of generality $$J=\{i,i+1,\cdots,i+l,j,j+1,\cdots,j+(k-l-1)\}.$$ There are only four possible members of this clique: $J\setminus\{i\}$, $J\setminus\{i+l\}$, $J\setminus\{j\}$ and $J\setminus\{j+(k-l-1)\}$. Any clique not containing an interval subset must be of one of these two forms.  

The final part follows from Section 9 of \cite{ops}.
\end{proof} 

Let a clique that does not contain an interval subset be an \emph{internal clique}, and one that does contain an interval subset a \emph{boundary clique}. From Lemma \ref{clique}, we know that an internal clique defined by a subset $J$ must have at most four members; these four subsets form the (hypothetical) \emph{square associated with $J$}. The subsets in any such square have a natural cyclic ordering. So for a given cluster $\mathcal{C}\in\mathcal{S}_{k,n}$ and non-interval subset $I\in\mathcal{C}$, we say a $k$-subset $I^\prime$ (not necessarily in $\mathcal{C}$) is an \emph{opposite of $I$} if there is an internal clique of $\mathcal{C}$, defined by a subset $J$, where $I$ and $I^\prime$ are at opposite points in the square associated with $J$. Note that $I^\prime$ does not have to be contained in $\mathcal{C}$. 

\begin{eg}
For example, take the cluster $\mathcal{C}$ from Figure \ref{wrecking}. There is an internal white clique $\mathcal{W}(\{3,6,7\})\in \mathcal{C}$. The square associated with $\{3,6,7\}$ has elements with labels $$2367, 3467, 3567, 3678.$$ This implies that the subsets with labels $2367$ and $3567$ are opposites, just as $3467$ and $3678$ are. In particular, the subset with label $3567$ is not in $\mathcal{C}$, but it has an opposite which is. 
\end{eg}
The following proposition is a key result that we may use to define superimposed triangulations.
\begin{prop}\label{clq}
Let $\mathcal{C}$ be a rectangular cluster in $\mathcal{S}_{k,n}$, and let $I=\{i,\cdots j-1,j,l,l+1\cdots,m\}\in \mathcal{C}$, where $(j-i)+(m-l) =k-2.$ Then:
\begin{itemize}
\item If $i<j\leq i+k-2$ then precisely one of the $k$-subsets with label \begin{align*}
&i\cdots (j-1),(l-1) \cdots m\text{ or}\\
&(i+1)\cdots j ,l \cdots (m+1)
\end{align*} is also in $\mathcal{C}$.
\item If $i\leq j< i+k-2$ then precisely one of the $k$-subsets with label 
\begin{align*}&(i-1) \cdots j,l \cdots (m-1)\text{ or}\\
&i \cdots (j+1),(l+1)\cdots m
\end{align*} is also in $\mathcal{C}$.
\end{itemize} 

\end{prop}

  \begin{proof}
Let $J_1$ be the subset with label $i\cdots (j-1),(l-1) \cdots m$ and $J_2$ the subset with label $(i+1)\cdots j ,l \cdots (m+1)$. Likewise let $K_1$ be the subset with label $i(i-1) \cdots j,l \cdots (m-1)$ and $K_2$ the subset with label $i \cdots (j+1),(l+1)\cdots m$
The pair of subsets $J_1$ and $J_2$ are crossing, just as the pair of subsets $K_1$ and $K_2$ are, so at most one of each pair may be in the cluster $\mathcal{C}$.

Observe that there are at least two and at most four (internal) white cliques in $\mathcal{C}$ containing $I$ - the four potential white cliques are $\mathcal{W}(I\cap J_1)$, $\mathcal{W}(I\cap J_2)$, $\mathcal{W}(I\cap K_1)$ and $\mathcal{W}(I\cap K_2)$. 
Assume there is a white clique $\mathcal{W}(I\cap J_1)$ in $\mathcal{C}$, and $J_1\notin \mathcal{C}$. Then this white clique in $\mathcal{C}$ has three members, $$\mathcal{W}(I\cap J_1)=(I, I\cap J_1\cup\{m+1\}, I\cap J_1\cup\{i-1\}).$$  
By Lemma \ref{cliky}, there is a black clique in $\mathcal{C}$ given by $\mathcal{B}(I\cup\{m+1\})$, which must have at least three members. Two members are given by $I$ and $I\cap J_1\cup\{m+1\}$; a third member could be either $J_2$ (which by assumption is not in $\mathcal{C}$) or the subset with label $i\cdots j,(l+1)\cdots (m+1)$, which is crossing with $I \cap J_1\cup\{i-1\}$, as $i-1<j<l<m+1$ modulo $n$. 

So we cannot have either white cliques $\mathcal{W}(I\cap J_1)$ or $\mathcal{W}(I\cap J_2)$ in $\mathcal{C}$. This means both $\mathcal{W}(I\cap K_1)$ and $\mathcal{W}(I\cap K_2)$ must be white cliques in $\mathcal{C}$. This implies, by the same argument as above, that either $K_1$ or $K_2$ is in $\mathcal{C}$. Let $K_1\in\mathcal{C}$, and $K_2\notin \mathcal{C}$. Then there is a black clique $\mathcal{B}(I\cup K_1)$ in $\mathcal{C}$, and there may only be one other black clique in $\mathcal{C}$ containing $I$. As $K_2\notin \mathcal{C}$, we have $I\cap K_2\cup\{i-1\}\in\mathcal{C}$ but  $I\cap K_2\cup\{m+1\}\notin \mathcal{B}(I\cup K_1)$. So the other black clique in $\mathcal{C}$ is of the form $\mathcal{B}(I\cup \{m+1\})$. But there can be no member of the white clique $\mathcal{W}(I\cap K_1)$ that also belongs to $\mathcal{B}(I\cup \{m+1\})$, a contradiction. This implies that either $J_1$ or $J_2$ is in $\mathcal{C}$. Similarly, either $K_1$ or $K_2$ is in $\mathcal{C}$.
\end{proof} 

Proposition \ref{clq} implies a structure for the subsets within each rectangular cluster in $\mathcal{S}_{k,n}$. Observe that for each rectangular cluster $\mathcal{C}\in\mathcal{S}_{k,n}$, the set of complements of members of $\mathcal{C}$ are a rectangular cluster in $\mathcal{S}_{n-k,n}$. 

\begin{defi}
Let $\mathcal{C}$ be a rectangular cluster in $\mathcal{S}_{k,n}$. Then the \emph{lattice interior} of $\mathcal{C}$ is constructed as follows:
For each $1\leq l \leq k-1$, the \emph{$l^\mathrm{th}$ row} is the collection of subsets in $\mathcal{C}$ of the form $\{i,i+1, \cdots, i+l-1, j,\cdots, j+k-l-2\}$.
For each $1\leq m\leq n-k-1$, the \emph{$m^\mathrm{th}$ column} is the collection of subsets in $\mathcal{C}$ whose complements form the $m^\mathrm{th}$ row.
\end{defi}

\subsection{Superimposed Triangulations}
Let $\mathcal{C}$ be a rectangular cluster, and consider the subset $I:=\{i\cdots j,l\cdots m\}\in\mathcal{C}$. Define the \emph{lower arc of $I$}, denoted by $\widecheck{I}$ to be the arc $(i,l)$. Likewise, defined the \emph{upper arc of $I$}, denoted by $\widehat{I}$ to be the arc $(j,m)$. In this section we study the properties of collections of upper and lower arcs.

Let $b$ and $r$ be positive integers. Then a \emph{$(b,r)$-tiling of the $n$-gon} is a tiling of the $n$-gon where one tile is a $(b+2)$-gon, one is an $(r+2)$-gon and the remaining tiles are triangles. 
Fix an integer $n$. Let $\mathcal{I}$ be a $(b,r)$-tiling and let $\mathcal{J}$ be a $(b+1,r-1)$-tiling. Then $\mathcal{I}$ and $\mathcal{J}$ are said to be \emph{non-crossing} if whenever an arc $\alpha$ in the tiling $\mathcal{I}$ intersects an arc $\beta$ in the tiling $\mathcal{J}$, then an endpoint of $\alpha$ and an endpoint of $\beta$ are adjacent to each other on the boundary of the $n$-gon. 

\begin{eg}\label{ball}
It should be kept in mind that $(b,r)$-tilings arise from rectangular clusters. For example, take the cluster from Figure \ref{wrecking}, there are three rows in the lattice interior:
\begin{align*}
&912,7;\ 123,7;\ 234,7;\ 345,7\\
&12,78;\ 23,78;\ 23,67;\ 34,67\\
&2,789;\ 2,678;\ 3,678;\ 4,678 
\end{align*}
By taking the collection of lower arcs, we obtain for each row the arcs:
\begin{align*}
&(9,7);\ (1,7);\ (2,7);\ (3,7)\\
&(1,7);\ (2,7);\ (2,6);\ (3,6)\\
&(2,7);\ (2,6);\ (3,6);\ (4,6) 
\end{align*}
These collections determine the following tilings, respectively a $(1,3)$-tiling, a $(2,2)$-tiling and a $(3,1)$-tiling.
$$
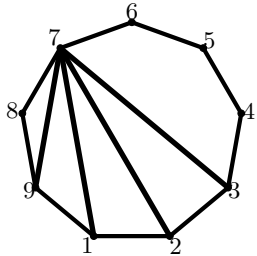
\begin{tikzpicture}[scale=0.5]
\draw [line width=1.6pt] (-4,2)-- (-2,2);
\draw [line width=1.6pt] (-2,2)-- (-0.4679111137620442,3.2855752193730785);
\draw [line width=1.6pt] (-0.4679111137620442,3.2855752193730785)-- (-0.12061475842818314,5.255190725397494);
\draw [line width=1.6pt] (-0.12061475842818314,5.255190725397494)-- (-1.120614758428183,6.9872415329663715);
\draw [line width=1.6pt] (-1.120614758428183,6.9872415329663715)-- (-3,7.671281819617709);
\draw [line width=1.6pt] (-3,7.671281819617709)-- (-4.879385241571816,6.987241532966372);
\draw [line width=1.6pt] (-4.879385241571816,6.987241532966372)-- (-5.879385241571816,5.255190725397496);
\draw [line width=1.6pt] (-5.879385241571816,5.255190725397496)-- (-5.532088886237956,3.28557521937308);
\draw [line width=1.6pt] (-5.532088886237956,3.28557521937308)-- (-4,2);
\draw [line width=2pt] (-5.532088886237956,3.28557521937308)-- (-4.879385241571816,6.987241532966372);
\draw [line width=2pt] (-4.879385241571816,6.987241532966372)-- (-4,2);
\draw [line width=2pt] (-4.879385241571816,6.987241532966372)-- (-2,2);
\draw [line width=2pt] (-4.879385241571816,6.987241532966372)-- (-0.4679111137620442,3.2855752193730785);
\begin{scriptsize}
\draw [fill=black] (-4,2) circle (2.5pt);
\draw[color=black] (-4.16891041165322,1.7619787629833805) node {$1$};
\draw [fill=black] (-2,2) circle (2.5pt);
\draw[color=black] (-1.8425412140050854,1.7412999256709527) node {$2$};
\draw [fill=black] (-0.4679111137620442,3.2855752193730785) circle (2.5pt);
\draw[color=black] (-0.30196783422920914,3.2094973748533313) node {$3$};
\draw [fill=black] (-0.12061475842818314,5.255190725397494) circle (2.5pt);
\draw[color=black] (0.08059065605070645,5.318738780720974) node {$4$};
\draw [fill=black] (-1.120614758428183,6.9872415329663715) circle (2.5pt);
\draw[color=black] (-0.963690628226901,7.20051297615191) node {$5$};
\draw [fill=black] (-3,7.671281819617709) circle (2.5pt);
\draw[color=black] (-2.9902166848448326,7.955290538055527) node {$6$};
\draw [fill=black] (-4.879385241571816,6.987241532966372) circle (2.5pt);
\draw[color=black] (-5.0374215787751915,7.262549488089193) node {$7$};
\draw [fill=black] (-5.879385241571816,5.255190725397496) circle (2.5pt);
\draw[color=black] (-6.143739374990083,5.329078199377188) node {$8$};
\draw [fill=black] (-5.532088886237956,3.28557521937308) circle (2.5pt);
\draw[color=black] (-5.699144372772883,3.2094973748533313) node {$9$};
\end{scriptsize}
\end{tikzpicture}
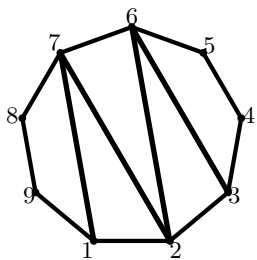
$$

$$\begin{tikzpicture}[scale=0.5]
\draw [line width=1.6pt] (-4,2)-- (-2,2);
\draw [line width=1.6pt] (-2,2)-- (-0.4679111137620442,3.2855752193730785);
\draw [line width=1.6pt] (-0.4679111137620442,3.2855752193730785)-- (-0.12061475842818314,5.255190725397494);
\draw [line width=1.6pt] (-0.12061475842818314,5.255190725397494)-- (-1.120614758428183,6.9872415329663715);
\draw [line width=1.6pt] (-1.120614758428183,6.9872415329663715)-- (-3,7.671281819617709);
\draw [line width=1.6pt] (-3,7.671281819617709)-- (-4.879385241571816,6.987241532966372);
\draw [line width=1.6pt] (-4.879385241571816,6.987241532966372)-- (-5.879385241571816,5.255190725397496);
\draw [line width=1.6pt] (-5.879385241571816,5.255190725397496)-- (-5.532088886237956,3.28557521937308);
\draw [line width=1.6pt] (-5.532088886237956,3.28557521937308)-- (-4,2);
\draw [line width=2pt] (-4.879385241571816,6.987241532966372)-- (-4,2);
\draw [line width=2pt] (-4.879385241571816,6.987241532966372)-- (-2,2);
\draw [line width=2pt] (-3,7.671281819617709)-- (-2,2);
\draw [line width=2pt] (-3,7.671281819617709)-- (-0.4679111137620442,3.2855752193730785);
\begin{scriptsize}
\draw [fill=black] (-4,2) circle (2.5pt);
\draw[color=black] (-4.16891041165322,1.7619787629833805) node {$1$};
\draw [fill=black] (-2,2) circle (2.5pt);
\draw[color=black] (-1.8425412140050854,1.7412999256709527) node {$2$};
\draw [fill=black] (-0.4679111137620442,3.2855752193730785) circle (2.5pt);
\draw[color=black] (-0.30196783422920914,3.2094973748533313) node {$3$};
\draw [fill=black] (-0.12061475842818314,5.255190725397494) circle (2.5pt);
\draw[color=black] (0.08059065605070645,5.318738780720974) node {$4$};
\draw [fill=black] (-1.120614758428183,6.9872415329663715) circle (2.5pt);
\draw[color=black] (-0.963690628226901,7.20051297615191) node {$5$};
\draw [fill=black] (-3,7.671281819617709) circle (2.5pt);
\draw[color=black] (-2.9902166848448326,7.955290538055527) node {$6$};
\draw [fill=black] (-4.879385241571816,6.987241532966372) circle (2.5pt);
\draw[color=black] (-5.0374215787751915,7.262549488089193) node {$7$};
\draw [fill=black] (-5.879385241571816,5.255190725397496) circle (2.5pt);
\draw[color=black] (-6.143739374990083,5.329078199377188) node {$8$};
\draw [fill=black] (-5.532088886237956,3.28557521937308) circle (2.5pt);
\draw[color=black] (-5.699144372772883,3.2094973748533313) node {$9$};
\end{scriptsize}
\end{tikzpicture}
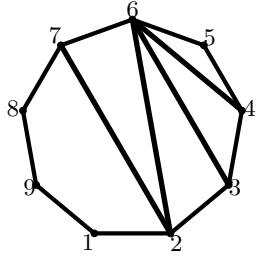
$$

$$\begin{tikzpicture}[scale=0.5]
\draw [line width=1.6pt] (-4,2)-- (-2,2);
\draw [line width=1.6pt] (-2,2)-- (-0.4679111137620442,3.2855752193730785);
\draw [line width=1.6pt] (-0.4679111137620442,3.2855752193730785)-- (-0.12061475842818314,5.255190725397494);
\draw [line width=1.6pt] (-0.12061475842818314,5.255190725397494)-- (-1.120614758428183,6.9872415329663715);
\draw [line width=1.6pt] (-1.120614758428183,6.9872415329663715)-- (-3,7.671281819617709);
\draw [line width=1.6pt] (-3,7.671281819617709)-- (-4.879385241571816,6.987241532966372);
\draw [line width=1.6pt] (-4.879385241571816,6.987241532966372)-- (-5.879385241571816,5.255190725397496);
\draw [line width=1.6pt] (-5.879385241571816,5.255190725397496)-- (-5.532088886237956,3.28557521937308);
\draw [line width=1.6pt] (-5.532088886237956,3.28557521937308)-- (-4,2);
\draw [line width=2pt] (-4.879385241571816,6.987241532966372)-- (-2,2);
\draw [line width=2pt] (-3,7.671281819617709)-- (-2,2);
\draw [line width=2pt] (-3,7.671281819617709)-- (-0.4679111137620442,3.2855752193730785);
\draw [line width=2pt] (-3,7.671281819617709)-- (-0.12061475842818314,5.255190725397494);
\begin{scriptsize}
\draw [fill=black] (-4,2) circle (2.5pt);
\draw[color=black] (-4.16891041165322,1.7619787629833805) node {$1$};
\draw [fill=black] (-2,2) circle (2.5pt);
\draw[color=black] (-1.8425412140050854,1.7412999256709527) node {$2$};
\draw [fill=black] (-0.4679111137620442,3.2855752193730785) circle (2.5pt);
\draw[color=black] (-0.30196783422920914,3.2094973748533313) node {$3$};
\draw [fill=black] (-0.12061475842818314,5.255190725397494) circle (2.5pt);
\draw[color=black] (0.08059065605070645,5.318738780720974) node {$4$};
\draw [fill=black] (-1.120614758428183,6.9872415329663715) circle (2.5pt);
\draw[color=black] (-0.963690628226901,7.20051297615191) node {$5$};
\draw [fill=black] (-3,7.671281819617709) circle (2.5pt);
\draw[color=black] (-2.9902166848448326,7.955290538055527) node {$6$};
\draw [fill=black] (-4.879385241571816,6.987241532966372) circle (2.5pt);
\draw[color=black] (-5.0374215787751915,7.262549488089193) node {$7$};
\draw [fill=black] (-5.879385241571816,5.255190725397496) circle (2.5pt);
\draw[color=black] (-6.143739374990083,5.329078199377188) node {$8$};
\draw [fill=black] (-5.532088886237956,3.28557521937308) circle (2.5pt);
\draw[color=black] (-5.699144372772883,3.2094973748533313) node {$9$};
\end{scriptsize}
\end{tikzpicture}$$

\end{eg}

It can be observed that the arcs $(2,6)$ and $(3,7)$ of Example \ref{ball} are crossing, although they arise as lower arcs from non-crossing $4$-subsets. This motivates the following defintion.

\begin{defn}\label{super}
Let $\mathcal{I}_1$ be a $(1,k-1)$-tiling of the $n$-gon, $\mathcal{I}_2$ a $(2,k-2)$-tiling of the $n$-gon and so on. 

The collection $\{\mathcal{I}_1,\mathcal{I}_2,\cdots, \mathcal{I}_{k-1}\}$ is a \emph{$(k-1)$-superimposed triangulation of the $n$-gon} if for each $1\leq j\leq k-1$, $\mathcal{I}_j$ and $\mathcal{I}_{j+1}$ are non-crossing .
\end{defn}

Then we have the main result of this section.

\begin{prop}\label{supertri}
Let $\mathcal{C}\in\mathcal{S}_{k,n}$ be a rectangular cluster. Then the collection of lower arcs $$\{\widecheck{I}|I\in\mathcal{C}\}$$ determines a $(k-1)$-superimposed triangulation. In addition, the collection of upper arcs $$\{\widehat{I}|I\in\mathcal{C}\}$$ also determines a $(k-1)$-superimposed triangulation.
\end{prop}

  \begin{proof}  It follows from the description in Proposition \ref{clq} that for each $1\leq l\leq k-1$, the collection of lower arcs for the $l^\mathrm{th}$ each row defines a $(l,k-l)$-tiling of the $n$-gon.
If $s<t<u<v$ modulo $n$ and $[s,u]$ is a lower arc in the $l$-th row and $(t,v)$ is a lower arc in the $(l+1)$-th row (or vice-versa), then for the two associated subsets to be non-crossing, we must have that either $s$ and $t$ are adjacent on the boundary of the $n$-gon, or $u$ and $v$ are adjacent on the boundary of the $n$-gon. This implies that the $(l,k-l)$-tiling and $(l+1,k-l-1)$-tiling arising from $\mathcal{C}$ are non-crossing. So the collection of lower arcs determines a $(k-1)$-superimposed triangulation.
Similarly, the collection of upper arcs defines $(k-1)$-superimposed triangulation. \end{proof}

\section{$\mathrm{SL}_3$-frieze patterns}

In this section, we will analyse $\mathrm{SL}_3$-frieze patterns more closely. Much of this theory is expected to generalise to higher dimensions, but some concepts do not generalise so easily. Especially, quiddity sequences are harder to understand in higher dimensions.

It has been observed in \cite[Section 3.2]{mgost}, see also \cite[Section 3.4]{mg}, that any $\mathrm{SL}_k$-frieze of width $n-k-1$ determines a point on the Grassmannian $\mathrm{Gr}(k,n)$. Conversely, it was shown in \cite[Theorem 3.1]{bfgst} that any cluster in $\mathcal{S}_{3,n}$ gives rise to an $\mathrm{SL}_3$-frieze pattern.
We also refer to \cite[Section 5]{mc} for an explanation of this connection.

\begin{defn}\label{dollazz}
Let $\mathcal{P}$ be a $\mathrm{SL}_k$-frieze pattern. Define the \emph{forwards quiddity element at $i$} to be $p_{(i-1)\widehat{i}(i+1)\cdots(i+k-1)}$ and likewise the \emph{reverse quiddity element at $i$} to be $p_{(i-k+1)\cdots(i-1)\widehat{i}(i+1)}$. A \emph{forwards (respectively reverse) quiddity sequence of order $n$} is the list of forwards (respectively reverse) quiddity elements at $i$ for $1\leq i\leq n$. 
\end{defn}

By Proposition \ref{supertri}, a rectangular cluster in $\mathcal{S}_{k,n}$ determines both a lower and an upper $2$-superimposed triangulation.
Let $\mathcal{C}\in\mathcal{S}_{k,n}$ be a rectangular cluster of quadrilateral type. Scott \cite[Theorem 1]{sco} proves that the $k$-subsets in $\mathcal{C}$ consist of all possible double-interval $k$-subsets of $\{1,2,\cdots, n\}$ whose lower arcs appear in the snake triangulation of a $n$-gon (having preassigned a cyclic ordering of the vertices of the $n$-gon).


\begin{eg}\label{eg1ii}

The snake triangulation of a hexagon may be labelled as follows:

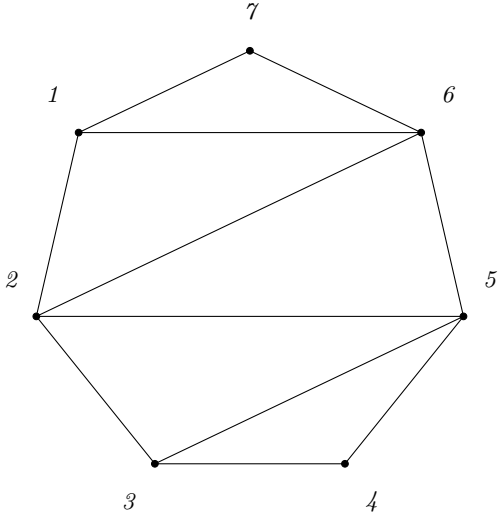
\begin{tikzpicture}[scale=2.5]
\draw (0.,4.)-- (-1.,4.);
\draw (-1.,4.)-- (-1.6234898018587332,4.7818314824680295);
\draw (-1.6234898018587332,4.7818314824680295)-- (-1.4009688679024186,5.756759394649853);
\draw (-1.4009688679024186,5.756759394649853)-- (-0.5,6.19064313376741);
\draw (-0.5,6.19064313376741)-- (0.40096886790241903,5.756759394649852);
\draw (0.40096886790241903,5.756759394649852)-- (0.6234898018587329,4.781831482468029);
\draw (0.6234898018587329,4.781831482468029)-- (0.,4.);
\draw (-1.4009688679024186,5.756759394649853)-- (0.40096886790241903,5.756759394649852);
\draw (0.40096886790241903,5.756759394649852)-- (-1.6234898018587332,4.7818314824680295);
\draw (-1.6234898018587332,4.7818314824680295)-- (0.6234898018587329,4.781831482468029);
\draw (0.6234898018587329,4.781831482468029)-- (-1.,4.);
\begin{scriptsize}
\draw [fill=black] (-1.,4.) circle (0.5pt);
\draw[color=black] (-1.14,3.8) node {3};
\draw [fill=black] (0.,4.) circle (0.5pt);
\draw[color=black] (0.14,3.8) node {4};
\draw [fill=black] (0.6234898018587329,4.781831482468029) circle (0.5pt);
\draw[color=black] (0.76,4.98) node {5};
\draw [fill=black] (0.40096886790241903,5.756759394649852) circle (0.5pt);
\draw[color=black] (0.54,5.96) node {6};
\draw [fill=black] (-0.5,6.19064313376741) circle (0.5pt);
\draw[color=black] (-0.5,6.39) node {7};
\draw [fill=black] (-1.4009688679024186,5.756759394649853) circle (0.5pt);
\draw[color=black] (-1.54,5.96) node {1};
\draw [fill=black] (-1.6234898018587332,4.7818314824680295) circle (0.5pt);
\draw[color=black] (-1.76,4.98) node {2};
\end{scriptsize}
\end{tikzpicture}

The double-interval subsets whose lower arcs coincide with one of the diagonals of this triangulation are:
\begin{align*}
&\{12,6\} &\{23,6\} && \{23,5\}\\
&\{2,67\} & \{2,56\} &&\{3,56\}
\end{align*}
which is a maximal collection of non-interval, pairwise non-crossing $3$-subsets of $\{1,2,\cdots,7\}$.

\end{eg}

Let $\mathcal{P}$ be the rectangular $(3,n)$-frieze pattern which arises from a rectangular cluster $\mathcal{C}\in\mathcal{S}_{k,n}$. Define the \emph{lower quiddity sequence of $\mathcal{P}$} to be the sequence of $\{1+L(i)|1\leq i\leq n\}$, where $L(i)$ is the number of lower arcs of $\mathcal{C}$ containing $i$. The \emph{upper quiddity sequence of $\mathcal{P}$} is defined similarly. 
\begin{eg}
The maximal collection of non-consecutive, pairwise non-crossing $3$-subsets of $\{1,2,\cdots,7\}$ from Example \ref{eg1ii} determines the following lower superimposed triangulation. The lower quiddity sequence is indicated by the numbers on the outside of the hexagon. 

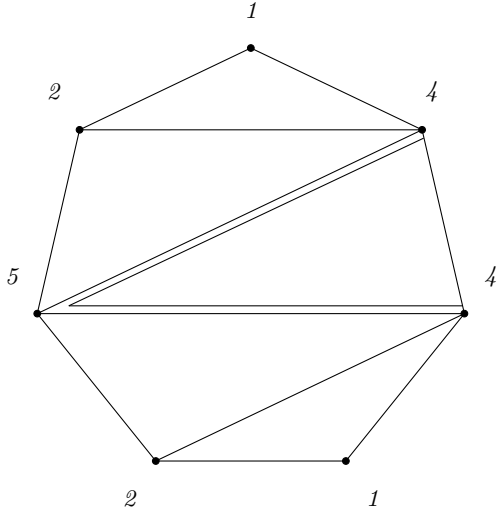
\begin{tikzpicture}[scale=2.5]
\draw (0.,4.)-- (-1.,4.);
\draw (-1.,4.)-- (-1.6234898018587332,4.7818314824680295);
\draw (-1.6234898018587332,4.7818314824680295)-- (-1.4009688679024186,5.756759394649853);
\draw (-1.4009688679024186,5.756759394649853)-- (-0.5,6.19064313376741);
\draw (-0.5,6.19064313376741)-- (0.40096886790241903,5.756759394649852);
\draw (0.40096886790241903,5.756759394649852)-- (0.6234898018587329,4.781831482468029);
\draw (0.6234898018587329,4.781831482468029)-- (0.,4.);
\draw (-1.4009688679024186,5.756759394649853)-- (0.40096886790241903,5.756759394649852);
\draw (-1.6234898018587332,4.7818314824680295)-- (0.6234898018587329,4.781831482468029);
\draw (0.6234898018587329,4.781831482468029)-- (-1.,4.);
\draw (0.40096886790241903,5.756759394649852)-- (-1.6234898018587332,4.7818314824680295);
\draw (-1.457580864995378,4.823200393460328)-- (0.4090179807712385,5.712522356626621);
\draw (-1.457580864995378,4.823200393460328)-- (0.6142461261173064,4.823200393460328);
\begin{scriptsize}
\draw [fill=black] (-1.,4.) circle (0.5pt);
\draw[color=black] (-1.14,3.8) node {2};
\draw [fill=black] (0.,4.) circle (0.5pt);
\draw[color=black] (0.14,3.8) node {1};
\draw [fill=black] (0.6234898018587329,4.781831482468029) circle (0.5pt);
\draw[color=black] (0.76,4.98) node {4};
\draw [fill=black] (0.40096886790241903,5.756759394649852) circle (0.5pt);
\draw[color=black] (0.45,5.96) node {4};
\draw [fill=black] (-0.5,6.19064313376741) circle (0.5pt);
\draw[color=black] (-0.5,6.39) node {1};
\draw [fill=black] (-1.4009688679024186,5.756759394649853) circle (0.5pt);
\draw[color=black] (-1.54,5.96) node {2};
\draw [fill=black] (-1.6234898018587332,4.7818314824680295) circle (0.5pt);
\draw[color=black] (-1.76,4.98) node {5};
\end{scriptsize}

\end{tikzpicture}
\end{eg}

If a cluster $\mathcal{C}\in\mathcal{S}_{k,n}$ is of quadrilateral type, we say the $\mathrm{SL}_3$-frieze pattern it determines is also of quadrilateral type.

\begin{theorem}\label{quid}
For a  given $\mathrm{SL}_3$-frieze pattern of quadrilateral type, its lower quiddity sequence is the same as its reverse quiddity sequence. Likewise, its upper quiddity sequence is the same as its forwards quiddity sequence.
\end{theorem}

  \begin{proof}
Let  $\mathcal{C}\in\mathcal{S}_{3,n}$ be the rectangular cluster that gives rise to an $\mathrm{SL}_3$-frieze pattern $\mathcal{P}$ that is of quadrilateral type. 
Consider the lower superimposed triangulation first - we calculate $L(i)$, the number of lower arcs that contain each vertex $i$. The cluster $\mathcal{C}$ has an explicit description as the subsets:
\begin{align*}
& \{i,i+1,n-i\} &&\forall 1\leq i<n/2-1\\
& \{i,i+1,n-i+1\} &&\forall 2\leq i<n/2-1\\
&\{i,n-i,n-i+1\} &&\forall 2\leq i<n/2-1\\
&\{i,n-i+1,n-i+2\} &&\forall 2\leq i<n/2
\end{align*}

First, it is clear that $\{i-1,i+1,i+2\}$ is crossing with $\{i,i+1,n-i\}$, $\{i,i+1,n+1-i\}$, $\{i,n-i,n-i+1\}$ or  $\{i,n-i+1,n-i+2\}$ if and only if they are themselves non-interval subsets. 

So we claim $p_{(i-1)(i+1)(i+2)}=L(i)+1$. This is certainly true if $L(i)=0$ - then $\{i-1,i+1,i+2\}\in\mathcal{C}$.

If $i$ is such that $L(i)=4$, then \begin{align*}
p_{(i-1)(i+1)(i+2)}&=p_{i(i+1)(n-i)}{p_{(i-1)(i+1)(i+2)}}\\&={p_{(i-1)(i+1)(n-i)}}p_{i(i+1)(i+2)}+p_{(i-1)i(i+1)}p_{(i+1)(i+2)(n-i)}\\&=p_{(i-1)(i+1)(n-i)}+1\\
p_{(i-1)(i+1)(n-i)}&=p_{i(i+1)(n-i+1)}{p_{(i-1)(i+1)(n-i)}}\\&={p_{(i-1)(i+1)(n-i+1)}}p_{i(i+1)(n-i)}+p_{(i+1)(n-i)(n-i+1)}p_{(i-1)i(i+1)}\\&=p_{(i-1)(i+1)(n-i+1)}+1\\
p_{(i-1)(i+1)(n-i+1)}&=p_{i(n-i)(n-i+1)}{p_{(i-1)(i+1)(n-i+1)}}\\&={p_{(i-1)(n-i)(n-i+1)}}p_{i(i+1)(n-i+1)}\\&+p_{(i-1)i(n-i+1)}p_{(i+1)(n-i)(n-i+1)}\\&=p_{i(i+1)(n-i+1)}+1\\
p_{i(i+1)(n-i+1)}&=p_{i(n-i+1)(n-i+2)}{p_{(i-1)(n-i)(n-i+1)}}\\&=p_{(i-1)(n-i+1)(n-i+2)}p_{i(n-i)(n-i+1)}\\&+p_{(n-i)(n-i+1)(n-i+2)}p_{(i-1)i(n-i+1)}\\&=2,
\end{align*}
This shows that $5=p_{(i-1)(i+1)(i+2)}=L(i)+1$. 
Now let $L(i)=3$, and assume without loss of generality that $i=n-1$. In this case $\{i,i+1,n-i\} =\{n-1,n,1\}$ is an interval, and a similar calculation to the above shows $p_{(i-1)(i+1)(i+2)}=4$. Continuing in this fashion shows that, if $L(i)=2$ or $L(i)=1$, then also $p_{(i-1)(i+1)(i+2)}=L(i)+1$, which shows the lower quiddity sequence is the same as the reverse quiddity sequence. Reversing the orientation shows that the upper quiddity sequence is the same as the forwards quiddity sequence. 
\end{proof}



\bibliographystyle{amsplain}

\bibliography{citat1ons}

\section{Acknowledgements}
This paper was completed as part of my PhD studies, with the support of the Austrian Science Fund (FWF): W1230.

\end{document}